\documentclass[11pt]{amsart}
\usepackage[colorlinks=true, pdfstartview=FitV, linkcolor=blue, 
citecolor=blue]{hyperref}

\usepackage{amssymb,bbm,amscd}
\usepackage{graphicx}
\usepackage{a4wide}

\theoremstyle{plain}
\newtheorem{theorem}{Theorem}[section]
\newtheorem{proposition}[theorem]{Proposition}
\newtheorem{lemma}[theorem]{Lemma}

\newtheorem{fact}[theorem]{Fact}

\theoremstyle{definition}
\newtheorem{remark}[theorem]{Remark}

\newtheorem{example}[theorem]{Example}

\numberwithin{equation}{section}
 
\newcommand{\dd}{\,\mathrm{d}}
\newcommand{\ii}{\ts\mathrm{i}}
\newcommand{\ee}{\,\mathrm{e}}
\newcommand{\ts}{\hspace{0.5pt}}
\newcommand{\nts}{\hspace{-0.5pt}}

\DeclareMathOperator{\card}{\mathrm{card}}
\DeclareMathOperator{\dist}{\mathrm{dist}}

\newcommand{\cD}{\mathcal{D}}

\newcommand{\cO}{\mathcal{O}}
\newcommand{\cP}{\mathcal{P}}
\newcommand{\scO}{{\scriptstyle\mathcal{O}}}
\newcommand{\ZZ}{\mathbb{Z}}
\newcommand{\RR}{\mathbb{R}\ts}
\newcommand{\CC}{\mathbb{C}\ts}

\newcommand{\NN}{\mathbb{N}}
\newcommand{\Nnull}{\NN^{}_{0}}

\newcommand{\XX}{\mathbb{X}}

\newcommand{\exend}{\hfill $\Diamond$}

\newcommand{\myfrac}[2]{\frac{\raisebox{-2pt}{$#1$}}
      {\raisebox{0.5pt}{$#2$}}}

\begin{document}

\title[Averaging Almost Periodic Functions]{Averaging Almost
  Periodic Functions\\[2mm] along Exponential Sequences}

\author{Michael Baake}
\address{Fakult\"{a}t f\"{u}r Mathematik,
  Universit\"{a}t Bielefeld,\newline \hspace*{\parindent}Postfach
  100131, 33501 Bielefeld, Germany}
\email{mbaake@math.uni-bielefeld.de}

\author{Alan  Haynes}
\address{Department of Mathematics, University of Houston, \newline
\hspace*{\parindent}3551 Cullen Blvd., Houston, TX 77204-3008, USA}
\email{haynes@math.uh.edu}

\author{Daniel Lenz}
\address{Fakult\"{a}t f\"{u}r Mathematik, Universit\"{a}t Jena, \newline
\hspace*{\parindent}Ernst-Abbe-Platz 2, 07743 Jena, Germany}
\email{daniel.lenz@uni-jena.de}
  
\begin{abstract}
  The goal of this expository article is a fairly self-contained
  account of some averaging processes of functions along sequences of
  the form $(\alpha^n x)^{}_{n\in\NN}$, where $\alpha$ is a fixed real
  number with $\lvert \alpha \rvert > 1$ and $x\in\RR$ is arbitrary.
  Such sequences appear in a multitude of situations including the
  spectral theory of inflation systems in aperiodic order. Due to the
  connection with uniform distribution theory, the results will mostly
  be \emph{metric} in nature, which means that they hold for
  \mbox{Lebesgue{\ts}}-almost every $x\in\RR$.
\end{abstract}

\maketitle
\thispagestyle{empty}

\section[Introduction]{Introduction}

A frequently encountered problem in mathematics and its applications
is the study of averages of the form
$\frac{1}{N} \sum_{n=1}^{N} f(x^{}_{n})$, where $f$ is a function with
values in $\CC$ or, more generally, in some Banach space, and
$(x^{}_{n})^{}_{n\in\NN}$ is a sequence of numbers in the domain of
$f$. Quite often, an exact treatment of these averages is out of hand,
and one resorts to the analysis of asymptotic properties for large
$N$.  This, for instance, is common in analytic number theory; compare
\cite{B-Hardy,B-Harman,B-Apostol} and references therein. Equally
important is the case where one can establish the existence of a limit
as $N\to\infty$, and then calculate it. This occupies a good deal of
ergodic theory, where Birkhoff's theorem and Kingman's subadditive
theorem provide powerful tools to tackle the problem; see
\cite{B-CFS,B-Walters} for background.

However, not all tractable cases present themselves in a way that is
immediately accessible to tools from ergodic theory. Also, depending on the
nature of the underlying problem, one might prefer a more elementary
method, as Birkhoff-type theorems already represent a fairly
advanced kind of `weaponry'. An interesting (and certainly not
completely independent) approach is provided by the theory of uniform
distribution of sequences, which essentially goes back to Weyl
\cite{B-Weyl} and has emerged as a major tool for the study of
function averages, in particular for functions that are periodic or
defined on a compact domain; see \cite{B-KN,B-DT,B-LP} and references
therein for more.

In this contribution, we recall some of these concepts, with an eye on
both methods (uniform distribution and ergodic theory), and use the
tools to treat averages of almost periodic functions along sequences
where this makes sense, in particular along sequences of the form
$(\alpha^{n} x)^{}_{n\in\NN}$ with `generic' $x\in\RR$ and a fixed
number $\alpha\in\RR$ with $\lvert \alpha \rvert > 1$. The first
subtlety that we shall encounter here emerges when $\alpha$ is not an
integer, which requires some care for functions that fail to be
locally Riemann-integrable. The second subtlety occurs when we extend
our considerations to almost periodic functions.

While the latter extension represents a relatively simple step beyond
periodic functions as long as one retains almost periodicity in the
sense of Bohr, matters become more involved when singularities occur
or weaker notions of almost periodicity are needed. Below, we shall
discuss some extensions of this kind that are relevant in practice;
compare \cite{B-FSS} for some related results. Let us note that some
of the notions and concepts used below are studied in much greater
generality in \cite{B-MS,B-NS}.

Before we begin our exposition, let us mention that averages of
$1$-periodic functions are often just the first step in the study of
Riesz--Raikov sums, that is, sums of the form
$\sum_{k=0}^{n-1} f( \alpha^{k} t)$.  Kac's
investigation for $\alpha=2$ in \cite{B-Kac} and Takahashi's refined
and generalised analysis \cite{B-Tak} are early examples that consider
limits (in a law of large numbers scaling) as well as distributions
(in a central limit theorem scaling, when
$\int_{0}^{1} f(t) \dd t = 0$). This led to a more elaborate
derivation of central limit theorems for Riesz--Raikov sums along
exponential sequences; compare \cite{B-Pet,B-Les,B-Rio} and references
therein.

Below, we are mainly interested in the Birkhoff-type averages, with a
focus on functions that fail to be periodic, but still have some
repetitivity structure in the form of a suitable almost periodicity.
In this sense, we have selected one particular aspect of Riesz--Raikov
sums that appears in the theory of aperiodic order
\cite{B-TAO,B-Sol,B-BG}.

\section{Preliminaries and general setting}

As far as possible, we follow the general (and fairly standard)
notation from \cite[Ch.~1]{B-TAO}, wherefore only deviations or
extensions will be mentioned explicitly. In particular, we will use
the Landau symbols $\cO$ and $\scO$ for the standard asymptotic
behaviour of real- or complex-valued functions; compare
\cite{B-Apostol,B-Hardy} for definitions and examples.

When two sets $A,B \subseteq \RR$ are given, we denote their
\emph{Minkowski sum} as
\[
   A+B \, := \, \{ a + b : a \in A , \, b \in B \} \ts .
\]
In particular, if the
point set $S \subset \RR$ is locally finite and $\varepsilon > 0$, we
use $S+(-\varepsilon,\varepsilon)$ for the open subset of $\RR$ that
emerges from $S$ as $\bigcup_{x\in S} (x-\varepsilon, x+\varepsilon)$.
Note that its complement in $\RR$ is then a closed set (possibly
empty).

Below, we frequently talk about results of metric nature, where
Lebesgue measure $\lambda$ on $\RR$ is our reference measure.
When a statement is true for almost every $x\in\RR$ with respect to
Lebesgue measure, we will simply say that it holds for a.e.\
$x\in\RR$. Likewise, when we speak of a null set, we mean a 
null set with respect to Lebesgue measure.

Recall that a sequence $(x_{n})^{}_{n\in\NN}$ of real numbers is
called \emph{uniformly distributed modulo} $1$ if, for all real
numbers $a,b$ with $0 \leqslant a < b \leqslant 1$, we have
\[
   \lim_{N\to\infty} \ts \myfrac{1}{N} \ts
   \card \Bigl( [a,b) \cap \big\{ 
   \langle x^{}_{1} \rangle , \ldots , \langle x^{}_{\nts N} \rangle 
   \big\} \Bigr)  \, = \: b-a  \ts ,
\]
where $\langle x \rangle$ denotes the fractional part\footnote{Since
  we use $\{ x \}$ for singleton sets, we resort to the less
  common notation $\langle x \rangle$ for the fractional part
  of $x$ in order to avoid misunderstandings.} of $x\in\RR$.  We refer
to \cite{B-KN,B-Bugeaud} for general background.  Recall that a
function $f$ on $\RR$ is $1$-\emph{periodic} if $f(x+1) = f(x)$ holds
for all $x\in\RR$.  One fundamental result, due to Weyl \cite{B-Weyl},
can now be formulated as follows; see also \cite[Thm.~5.3]{B-Harman}.

\begin{lemma}[\textsf{Weyl's criterion}]\label{B-lem:Weyl}
  For a sequence\/ $(x_{n})^{}_{n\in\NN}$ of real numbers,
  the following properties are equivalent.
\begin{enumerate}\itemsep=2pt
\item The sequence is uniformly
     distributed modulo\/ $1$.
\item For every complex-valued, $1\nts$-periodic continuous function\/
  $f\nts$, one has
\[
   \lim_{N\to\infty} \myfrac{1}{N}
   \sum_{n=1}^{N} f (x_{n}) \, =  \int_{0}^{1}
    \! f(x) \dd x \ts .
\]
\item The relation from\/ $(2)$ holds for every\/
 $1\nts$-periodic function that is locally Riemann-integrable.
\item The relation\vspace{-2mm}
\[
    \lim_{N\to\infty} \myfrac{1}{N}
   \sum_{n=1}^{N} \ee^{2 \pi \ii k \ts x_{\nts n}} 
    \, = \, \delta^{}_{k,0}
\]
holds for every\/ $k\in\ZZ$.  \qed
\end{enumerate}
\end{lemma}

Let us note in passing that the equivalence of conditions (1) and (2)
can also be understood in terms of systems of almost invariant
integrals and Eberlein's ergodic theorem. These notions are reviewed
and studied in some detail in \cite{B-NS}.

\begin{remark}
  Weyl's criterion (which is also known as
  Weyl's lemma) is an important tool for calculating the average of a
  locally Riemann-integrable periodic function along a uniformly
  distributed sequence.  In fact, a $1$-periodic function is locally
  Riemann-integrable if and only if the Birkhoff average converges for
  \emph{every} sequence that is uniformly distributed modulo $1$;
  compare \cite{B-dBP} as well as \cite[p.~123]{B-Harman}.
  
  Conversely, the integral of a Riemann-integrable function can be
  approximated by averages along uniformly distributed sequences.
  This is a standard method in numerical integration, in particular
  for higher-dimensional integrals; see \cite{B-Hart,B-LP} and
  references therein for more.  \exend
\end{remark}

There is an abundance of known results on uniformly distributed
sequences and their finer properties; we refer to \cite{B-KN} for the
classic theory and to \cite{B-Bugeaud} and references therein for more
recent developments. Here, we are particularly interested in one
specific class of sequences, for which the uniform distribution is
well known; compare \cite[Thms.~1.7 and 1.10]{B-Bugeaud} as well as
\cite[Sec.~7.3, Thm.~1]{B-CFS} or \cite[Cor.~1.4.3 and
  Exs.~1.4.3]{B-KN}.

\begin{fact}\label{B-fact:power-dist}
  Consider the sequence\/ $( \alpha^{n} x )^{}_{n\in\NN}$.  For
  fixed\/ $\alpha\in\RR$ with\/ $\lvert \alpha \rvert > 1$, it is
  uniformly distributed modulo\/ 
  $1$ for a.e.\/ $x\in\RR$.  For
  fixed\/ $0 \ne x \in \RR$, the sequence is uniformly distributed modulo\/ $1$
  for a.e.\/ $\alpha\in\RR$ with\/ $\lvert \alpha \rvert > 1$.  \qed
\end{fact}

Below, we will mainly be concerned with the first case, where a fixed
$\alpha$ with $\alpha>1$ or $\lvert \alpha \rvert > 1$ is given. This
situation is of particular interest in the theory of aperiodic
order, for instance in connection with the
renormalisation analysis of inflation tiling systems, because it plays
an important role for the averaging of functions with certain
repetition properties along the real line.

Let us mention in passing that, when $\alpha = q \geqslant 2$ is an integer,
$(q^{n} x)^{}_{n\in\NN}$ is uniformly distributed modulo\/ $1$ if and
only if $x$ is a \emph{normal} number
\cite{B-Bugeaud} in base $q$, which means that the $q$-ary expansion
of $x$ contains all possible finite substrings in the digit set
$\{ 0, 1, \ldots, q-1\}$ in such a way that any substring of length
$\ell$ has frequency $1/q^{\ell}$. In the Lebesgue sense, a.e.\
$x\in\RR$ is normal with respect to all integer bases
\cite[Thm.~4.8]{B-Bugeaud}, but it is a hard problem to decide on
normality for any given number.

\begin{remark}\label{B-rem:gen-u}
  Consider a sequence $(u^{}_{n})^{}_{n\in\NN_{0}}$ of real numbers
  such that $\inf_{n\ne m} \,\lvert u^{}_{n} - u^{}_{m} \rvert >
  0$. Then, by \cite[Cor.~1.4.3]{B-KN}, the sequence
  $(u^{}_{n} x)^{}_{n\in\NN_{0}}$ is uniformly distributed modulo $1$
  for a.e.\ $x\in\RR$. In fact, it is a rather direct consequence
  that, for any $k \in \NN$, $\ell\in\NN_{0}$ and any real number
  $L>0$, the arithmetic progression sequence
  $(u^{}_{k m + \ell} \, x)^{}_{m\in\NN}$ is uniformly distributed
  modulo $L$ for a.e.\ $x\in\RR$. This is the \emph{total Bohr
    ergodicity} of  the sequence $(u^{}_{n})^{}_{n\in\NN_{0}}$ as introduced in
  \cite[Def.~2.1]{B-FSS}.  Clearly, $u^{}_{n} = \alpha^{n}$ with
  $\lvert \alpha \rvert > 1$ defines such a sequence, while no bounded
  sequence can have this property.  \exend
\end{remark}

As soon as we leave the realm of periodic functions that are locally
Riemann-integrable, the desired averaging statements will need some
finer properties of our sequences $(\alpha^n x)^{}_{n\in\NN}$, where
we assume $\lvert \alpha \rvert > 1$ as before. In particular, we will
need details on the uniform distribution (or the deviation from it)
and some information on the approximation or non-approximation of
numbers in a given set by the sequence elements.  For the first issue,
we need the {discrepancy} structure of the sequence, and some
{Diophantine approximation} properties for the latter.

Recall that the \emph{discrepancy} of a sequence
$(x_{n})^{}_{n\in\NN}$ is quantified in terms of the first $N$
elements of the sequence (taken modulo $1$), namely by the number
\[
  \cD^{}_{\nts N} \, := \, \sup_{0\leqslant a < b \leqslant 1}
  \left| \myfrac{1}{N} \ts \card \Bigl( [a,b) \cap
  \big\{ \langle x^{}_{1} \rangle, \ldots , \langle x^{}_{\nts N} \rangle 
  \big\}  \Bigr) - (b-a) \right| \ts ,
\]
together with its asymptotic properties as $N \to \infty$. 

In this terminology, the sequence $(x_{n})^{}_{n\in\NN}$ is uniformly
distributed modulo $1$ if and only if
$\lim_{N\to\infty} \cD^{}_{\nts N} = 0$; see
\cite[Thm.~1.12]{B-Bugeaud}.  The discrepancy of sequences is a huge
research area in itself; see \cite{B-DT} and references therein for a
comprehensive exposition. We will need the following result, where we
refer to \cite[Thm.~5.13]{B-Harman} and \cite{B-KHN} for proofs.

\begin{fact}\label{B-fact:discrep}
  Let\/ $\alpha \in \RR$ with\/ $\lvert \alpha \rvert > 1$ be
  given. Then, for any fixed\/ $\varepsilon>0$, the discrepancy of the
  sequence\/ $(\alpha^{n} x)^{}_{n\in\NN}$, for a.e.\/ $x\in\RR$, 
   asymptotically is
\[
    \cD^{}_{\nts N} \, = \, \cO \left(
    \frac{\bigl(\log (N)\bigr)^{\frac{3}{2} + \varepsilon}}
    {\sqrt{N \ts }}  \right)\vspace{-2mm}
\]
   as\/ $N\to\infty$.  \qed
\end{fact}

Next, we need a \emph{Diophantine approximation} property.  If
$\varnothing \ne Y\subset\RR$ is a uniformly discrete point set,
compare \cite[Sec.~2.1]{B-TAO}, we can define
\[
   \dist (x, Y) \, := \, \min_{y\in Y}\, \lvert x - y \rvert 
\]
as the distance of $x \in \RR$ from $Y$. Now, one can state 
the following metric `non-approximation' result, which is a
versatile generalisation of the classic situation with 
$Y\! =\ZZ$.

\begin{lemma}\label{B-lem:Dio}
  Let\/ $\alpha \in \RR$ with\/ $\lvert \alpha \rvert > 1$ be given,
  and let\/ $Y\nts\nts\subset \RR$ be a non-empty, uniformly discrete
  point set. Further, fix some\/ $\varepsilon > 0$.  Then, for a.e.\/
  $x\in\RR$, the inequality
\[
    \dist (\alpha^{n-1} x, Y) \, \geqslant \,
    \myfrac{1}{n^{1+\varepsilon}}
\]
  holds for almost all\/ $n\in\NN$, by which we mean that it
  holds for all natural numbers except at most finitely many.
\end{lemma}

\begin{proof}
The statement is trivial when $Y$ is a finite point set, so let us
assume that $Y$ is unbounded. In this case, one still has
\[
     \delta \, := \, \inf \big\{ \lvert x - y \rvert :
     x,y \in Y \ts , \; x \ne y \big\} \, > \, 0 \ts ,
\]
due to the assumed uniform discreteness of $Y$. Consequently,
the number of points of $Y$ in an arbitrary interval $[a,b]$
with $a \leqslant b$ satisfies
\begin{equation}\label{B-eq:card-bound}
   \card \bigl( Y \cap [a,b] \bigr) \, \leqslant \,
   1 + \Bigl[ \myfrac{b-a}{\delta} \Bigr] ,
\end{equation}
where $[\ts . \ts ]$ is the Gau{\ss} bracket.

Let $m\in\ZZ$ be arbitrary, but fixed, and consider
$I_{m} = [\ts m, m \! + \! 1 ]$.  With
$\RR = \bigcup_{m\in\ZZ} I_{m}$, it suffices to show that our claim
fails at most for a null set within the interval $I_{m}$, as the
countable union of null sets is still a null set.

Choose $\varepsilon >0$ and, for $n\in\NN$, consider the set
\[
     A^{(m)}_{n} \, := \, \Big\{ x \in I_{m} : \dist 
    (\alpha^{n-1}x,Y) < \myfrac{1}{n^{1+\varepsilon}} \Big\} .
\]
It is clearly measurable, and its measure, since
$\lvert \alpha \rvert > 1$, can be estimated as
\[
\begin{split}
   \lambda \bigl( A^{(m)}_{n} \bigr) \, & = \,
   \myfrac{1}{\lvert\alpha\rvert^{n-1}} \, \lambda 
   \Big\{ z \in \alpha^{n-1} I_{m} : \dist (z, Y) <
   \myfrac{1}{n^{1+\varepsilon}} \Big\} \\[1mm]
   & \leqslant \, \myfrac{1}{\lvert\alpha\rvert^{n-1}} 
   \, \myfrac{2}{n^{1+\varepsilon}} 
   \left( 1 + \Bigl[\myfrac{\lvert\alpha\rvert^{n-1}}
           {\delta} \Bigr] \right)
   \, = \: \cO \Bigl( \myfrac{1}{n^{1+\varepsilon}} \Bigr),
\end{split}
\]
where the second step is a consequence of Eq.~\eqref{B-eq:card-bound}.
We thus know that there is a $C>0$ such that
$\lambda \bigl( A^{(m)}_{n} \bigr) \leqslant C/n^{1+\varepsilon}$ for all
$n\in\NN$.  

Now, we have
\[
    0 \, \leqslant \, \sum_{n\geqslant 1} \lambda 
   \bigl( A^{(m)}_{n} \bigr) \, \leqslant \, 
   C \ts \sum_{n\geqslant 1} \myfrac{ 1}{n^{1+\varepsilon}} \ts ,
\]
where the second sum is convergent, and thus also the first.
Then, Cantelli's lemma\footnote{This being the `easy half' of the
  Borel--Cantelli lemma, which goes back to Cantelli, we follow
  \cite[App.~C]{B-Bugeaud} in our terminology, and also refer
  to this reference for a proof.} tells us that
\[
    E^{(m)}_{\infty} \, := \, \bigl\{ x \in I_{m} :
    x \in A^{(m)}_{n} \text{ for infinitely many } 
    n \in \NN \bigr\}
\]
is indeed a null set, which is what we needed to show.
\end{proof}

\begin{remark}
  Though immaterial for the proof, it is often useful in an
  application to also remove all $x\in\RR$ with
  $Y\cap \{ \alpha^{n-1} x : n\in\NN\} \ne \varnothing$, which constitutes a
  null set because it is clearly countable or even finite.  \exend
\end{remark}

The lower bound in Lemma~\ref{B-lem:Dio} can be replaced by the values
of a more general, non-negative arithmetic function, $\psi (n)$ say,
provided one also has the summability condition $\sum_{n\in\NN} \psi(n) <
\infty$. When this sum diverges, the situation changes. Indeed, for
instance if $\alpha = 2$ and $Y=\ZZ$, there is then a set $X \subset
\RR$ of full measure such that, for $x\in X$, the distance of $2^{n-1}
x$ from the nearest integer is smaller than $\frac{1}{n}$ for
infinitely many $n\in\NN$; see \cite{B-BaHa} for a more general result
in this direction.  Moreover, one cannot do better than using some
$\varepsilon>0$ in Lemma~\ref{B-lem:Dio}, in line with the divergence
of the harmonic series.

\section{Averaging periodic functions}

Let us first state a result that emerges from an application of
Weyl's criterion to the special type of sequences we are
interested in.

\begin{fact}\label{B-fact:per-sample}
  Let\/ $ f \! : \, \RR \xrightarrow{\quad} \CC$ be a continuous or,
  more generally, a locally Riemann-integrable function that is\/ 
  $L$-periodic, so\/ $f(x+L) = f(x)$ holds for some fixed\/ $L>0$ 
  and all\/ $x\in\RR$. If\/ $\alpha$ is a real number with\/
  $\lvert \alpha \rvert > 1$, one has
\[
    \lim_{N\to\infty} \myfrac{1}{N} \sum_{n=0}^{N-1} f(\alpha^{n} x)
    \, = \, \myfrac{1}{L} \int_{0}^{L} \! f(y) \dd y\vspace{-1mm}
\]
  for a.e.\/ $x\in\RR$.
\end{fact}

\begin{proof}
  Since any $L$-periodic continuous function is also locally
  Riemann-integrable, it suffices to consider the latter class.
  Define a new function $g$ by $g(x) := f(L\ts x)$, which clearly is
  $1$-periodic and locally Riemann-integrable. Now, we have
\[
   \myfrac{1}{N} \sum_{n=0}^{N-1} f(\alpha^{n} x) \, = \,
   \myfrac{1}{N} \sum_{n=0}^{N-1} g \bigl(\alpha^{n} 
   \tfrac{x}{L} \bigr) ,
\]
where $\bigl(\alpha^{n} \tfrac{x}{L} \bigr)_{n\in\NN}$ is uniformly
distributed modulo $1$ for a.e.\ $\frac{x}{L}\in\RR$, and hence also
for a.e.\ $x\in\RR$, by Fact~\ref{B-fact:power-dist}. Consequently,
Weyl's criterion from Lemma~\ref{B-lem:Weyl}
tells us that
\[
   \myfrac{1}{N} \sum_{n=0}^{N-1} f(\alpha^{n} x)
   \; \xrightarrow{ \, N \to \infty \,} \,
   \int_{0}^{1} g (z) \dd z \, = \, 
   \myfrac{1}{L} \int_{0}^{L} \! f(y) \dd y
\]
holds for all such cases, which means for a.e.\ $x\in\RR$
as claimed.
\end{proof}

Note that one can rewrite Fact~\ref{B-fact:per-sample} with
the \emph{mean} of $f$, because
\[
   M (f) \: := \ts \lim_{T\to\infty} \myfrac{1}{2 \ts T}
   \int_{-T}^{T}  f(y) \dd y \: = \ts \lim_{T\to\infty} 
   \myfrac{1}{T}  \int_{a}^{a+T} \! \! f(y) \dd y
   \, = \, \myfrac{1}{L} \int_{0}^{L} \! f(y) \dd y
\]
holds for every $L$-periodic function that is locally
Riemann-integrable, where the limit clearly is uniform in $a \in \RR$.

\begin{example}\label{B-ex:trig-mono}
  Fix $k\in\RR$ and consider the trigonometric monomial defined by
  $\psi^{}_{k} (x) = \ee^{2 \pi \ii k x}$. Unless $k=0$, in which case
  $\psi^{}_{0} \equiv 1$, the function $\psi^{}_{k}$ has period
  $\frac{1}{\lvert k \rvert} > 0$. For $\alpha\in\RR$ with
  $\lvert \alpha \rvert >1$, Fact~\ref{B-fact:per-sample} implies that
\[
    \myfrac{1}{N} \sum_{n=0}^{N-1} \psi^{}_{k} (\alpha^n x)
    \; \xrightarrow{\, N \to \infty \,} \; M(\psi^{}_{k})
    \, = \, \begin{cases} 1 , &  k = 0, \\
        0, & \text{otherwise}, \end{cases}
\]
  holds for a.e.\ $x\in\RR$.
  
  More generally, if $(u^{}_{n})^{}_{n\in\NN_{0}}$ with
  $\inf_{n\ne m} \, \lvert u^{}_{n} - u^{}_{m} \rvert > 0$
  is the type of sequence from Remark~\ref{B-rem:gen-u},
  the above convergence statement also holds with 
  $\alpha^{n} x$ replaced by $u^{}_{n} \ts x$.
\exend
\end{example}

It is clear from the proof of Fact~\ref{B-fact:per-sample} that, for
periodic functions, it suffices to consider the case $L=1$ without
loss of generality, as we do from now on. Our next step shows that,
for $\alpha\in\ZZ$, one can go beyond the class of $1$-periodic
functions that are locally Riemann-integrable.

\begin{lemma}\label{B-lem:per-sample}
  Consider a function\/ $f \in L^{1}_{\mathrm{loc}} (\RR)$ that is\/
  $1$-periodic. Fix\/ $q\in\ZZ$ with\/ $\lvert q \rvert \geqslant 2$.
  Then, for a.e.\/ $x \in \RR$, one has
\[
    \myfrac{1}{N} \sum_{n=0}^{N-1} \! f (q^n x)
    \, \xrightarrow{\, N \to \infty \,} 
    \int_{0}^{1} \! f (y) \dd y \, = \, M (f) \ts .
\]
\end{lemma}

\begin{proof}
  Since $q\in\ZZ$, we may view the average as a Birkhoff sum for the
  dynamical system on $[0,1]$ defined by the mapping
  $x \mapsto q \ts x \bmod 1$. It is well known that Lebesgue measure
  is invariant and ergodic for this system, compare \cite{B-Cigler}
  and references therein, wherefore we may employ Birkhoff's ergodic
  theorem \cite{B-Walters} to $f$,
  which is Lebesgue{\ts}-integrable on $[0,1]$ by assumption, and our
  claim follows.
\end{proof}

Note that the exceptional set, for which the limit differs or does not
exist, may depend on $f$ when the latter fails to be continuous. In
fact, there clearly is no uniformly distributed sequence that will
work for \emph{all} $1$-periodic $f \in L^{1}_{\mathrm{loc}} (\RR)$.  Still,
the result of Lemma~\ref{B-lem:per-sample} suggests that something
more general than Fact~\ref{B-fact:per-sample} might also be true when
our multiplier $\alpha$ fails to be an integer. However, we cannot
apply the `trick' with Birkhoff's ergodic theorem when
$\alpha \not\in \ZZ$. This is due to the fact that the sequence
$(\langle \alpha^n x \rangle)^{}_{n\in\Nnull}$, which is uniformly
distributed on $[0,1)$ for a.e.\ $x\in\RR$ by
Fact~\ref{B-fact:power-dist}, does no longer agree with the orbit of
$x$ under the mapping $T$ defined by $x \mapsto \alpha x \bmod{1}$.
The latter, for a.e.\ $x\in\RR$, follows the distribution of the
(ergodic) R\'{e}nyi--Parry measure \cite{B-Ren,B-Par} for $\alpha$,
which is of the form $h_{\alpha} \lambda$ with $h_{\alpha}$ being
Lebesgue{\ts}-integrable on $[0,1)$. When $\alpha\not\in\ZZ$, the
measures $\lambda$ and $h_{\alpha} \lambda$ are still equivalent as
measures, but different; see \cite{B-Cigler} and references therein
for more.

\begin{example}
To illustrate the difference, consider $\alpha = \tau = \frac{1}{2}
\bigl( 1 + \sqrt{5}\,\bigr)$, which is one of the simplest examples
in this context. When $f$ is $1$-periodic and locally Riemann-integrable,
we get
\[
    \myfrac{1}{N} \sum_{n=0}^{N-1} f(\tau^n x)
    \, \xrightarrow{\, N\to\infty\,} 
    \int_{0}^{1} f(x) \dd x \, = \, M(f)
\]
for a.e.\ $x\in\RR$ by Weyl's criterion (Lemma~\ref{B-lem:Weyl}).

In comparison, let $T$ be defined by $x \mapsto \tau x \bmod{1}$ 
on $[0,1)$. Then, for a.e.\ $x\in [0,1)$, the orbits
$(T^n x)^{}_{n\in\Nnull}$ follow the distribution given by
the piecewise constant function \cite[Ex.~4]{B-Ren}
\[
    h_{\tau} (x) \, = \, \begin{cases}
    \frac{5 + 3 \sqrt{5}}{10} , &   0 \leq x < \frac{1}{\tau} , \\[1mm]
    \frac{5 + \sqrt{5}}{10}, & \frac{1}{\tau} \leq x < 1 .
    \end{cases}
\]
Since $T$ is ergodic for the measure $h_{\tau}\lambda$, Birkhoff's
theorem tells us that, for any \mbox{Lebesgue{\ts}}-integrable function
$f$ on $[0,1)$, one has
\[
   \lim_{N\to\infty} \myfrac{1}{N}
   \sum_{n=0}^{N-1} f(T^n x) \, = \int_{0}^{1} f(x) \, 
        h_{\tau} (x) \dd x
\]
for a.e.\ $x\in[0,1)$, and this limit will generally differ from $M(f)$.

Moreover, since the sequences
$( \langle \tau^n x \rangle )^{}_{n\geq 0}$ and
$(T^n x)^{}_{n\geq 0}$ are not easily relatable, one cannot infer the
convergence of averages along the exponential sequence from those
along the orbits under $T$.  \exend
\end{example}

Let us now extend Fact~\ref{B-fact:per-sample} beyond
Riemann-integrable functions by stating one version of Sobol's
theorem \cite[Thm.~1]{B-Sobol}.

\begin{theorem}\label{B-thm:p-Sobol}
  Let\/ $\alpha\in\RR$ with\/ $\lvert \alpha \rvert > 1$ be fixed, and
  consider a\/ $1$-periodic function\/ $f \in L^{1}_{\mathrm{loc}} (\RR)$
  that fails to be locally Riemann-integrable.  Assume that there is a
  finite set\/ $F\subset [0,1]$ such that\/ $f$, for every\/
  $\delta >0$, is Riemann-integrable on the complement of\/
  $F+(-\delta,\delta)$ in\/ $[0,1]$. Assume further that, for every\/
  $z\in F$, there is a\/ $\delta_{z}>0$ such that\/ $f$ is
  differentiable on the punctured interval\/
  $(z-\delta_{z}, z+\delta_{z}) \setminus \{ z \}$ and that, for any\/
  $s>0$,
\[
   V^{}_{\! N} (z,s) \, := \int_{z-\delta_{\nts z}}^{z-\frac{1}{N^{\nts s}}}
   \lvert f' (x) \rvert \dd x \, + 
   \int_{z+\frac{1}{N^{\nts s}}}^{z+\delta_{\nts z}}
   \lvert f' (x) \rvert \dd x \, = \, \scO
   \bigl( N_{\vphantom{I}}^{\frac{s}{2} - \eta} \bigr)
\]
  holds for some\/ $\eta = \eta(z) > 0$ as\/ $N\to\infty$.

Then, for a.e.\/ $x \in \RR$, one has
\[
    \myfrac{1}{N} \sum_{n=0}^{N-1} \! f (\alpha^n x)
    \, \xrightarrow{\, N \to \infty \,} 
    \int_{0}^{1} \! f (y) \dd y \, = \, M (f) \ts .
\]
\end{theorem}

\begin{proof}[Sketch of proof]
  Since $F$ is finite, we may choose
  $0 < \delta \leqslant \min_{z\in F} \, \delta_{z}$ small enough such that
  the open sets $(z-\delta, z+\delta)$ with $z\in F$ are disjoint. By
  writing $f$ as a sum of a locally Riemann-integrable function (such
  as the restriction $f^{(\delta)}$ of $f$ to the complement of
  $\ZZ + F + (-\delta,\delta)$) and $r = \card (F)$ `problematic'
  terms, the latter supported on $(z-\delta, z+\delta)$ with $z\in F$,
  it is clear that our claim follows if we can deal with one of these
  problematic terms. So, select one $z\in F$.  Without loss of
  generality, we may assume that $(z-\delta, z+\delta) \subset [0,1]$,
  as we can otherwise shift the unit interval because $f$ and
  $f^{(\delta)}$ are $1$-periodic.

  One can now repeat the original proof from \cite{B-Sobol}, or the
  more extensive version in \cite[Sec.~2]{B-Hart}. Here, the validity
  of the convergence claim emerges from the observation that, for
  a.e.\ $x\in\RR$, the number $\langle \alpha^{n-1} x \rangle$ does
  not come closer to $z$ than $1/n^{1+\varepsilon}$, for any fixed
  $\varepsilon > 0$ and then all $n\in\NN$ except at most finitely
  many. This follows from Lemma~\ref{B-lem:Dio} with $Y = z+\ZZ$.
  Now,
\[
   V^{}_{\! N} (z,1+\varepsilon) \, = \, \scO \Bigl(
   N_{\vphantom{I}}^{\frac{1+\varepsilon}{2} - \eta}\Bigr)
\]
for some $\eta > 0$ by assumption. Since $\eta$ does not depend on
$\varepsilon$, we are still free to choose $\varepsilon > 0$ small
enough so that $\vartheta := \eta - \frac{\varepsilon}{2} > 0$.

Now, the potentially large contribution to our averaging sum from
sequence elements close to $z$ are properly `counterbalanced' by the
discrepancy of $(\alpha^{n} x)^{}_{n\in\NN}$, where
we invoke Fact~\ref{B-fact:discrep} with the $\varepsilon$ just
chosen.  One obtains
\begin{equation}\label{B-eq:bound}
    \cD^{}_{\nts N} \cdot V^{}_{\! N} (z, 1 +\varepsilon) \, = \,
    \scO \biggl( \frac{\bigl(\log (N)\bigr)^{\frac{3}{2} + 
    \varepsilon}} {N^{\vartheta}}\biggr) \, = \, \scO (1) \ts ,
\end{equation}
which is a sufficient criterion for the claimed convergence because
\[
    \myfrac{1}{N} \sum_{n=0}^{N-1} f^{(\delta)} (\alpha^n x)
    \, \xrightarrow{\,N\to\infty\,} \, M \bigl( f^{(\delta)}\bigr)
\]
holds for a.e.\ $x\in\RR$, while the Birkhoff average of
$f-f^{(\delta)}$ is controlled by Eq.~\eqref{B-eq:bound} and tends to
$0$ as $\delta\searrow 0$.
\end{proof}

\begin{remark}\label{B-rem:min-eta}
  The assumption that $F$ in Theorem~\ref{B-thm:p-Sobol} is a finite
  set implies $\delta := \min_{z\in F}\, \delta_{z} >0$ as well
  as $\min_{z\in F}\, \eta(z) >0$. Later, we will replace this setting
  by a suitable compactness assumption to extend the result of this
  theorem to almost periodic functions.  \exend
\end{remark}

\begin{remark}\label{B-rem:bad}
  The differentiability assumption for $f$ near the `bad' points is
  convenient, but not necessary. It can be replaced by the requirement
  that the total variation of $f$ on sets of
  the form $(z-\delta, z-N^{-s} ] \cup [z+N^{-s}, z+\delta)$ behaves
    as stated for $V^{}_{\! N} (z,s)$; compare \cite{B-Sobol,B-Hart}.
    \exend
\end{remark}

As mentioned earlier, results of this type are also of interest for
the numerical calculation of integrals, for instance with methods of
(quasi-) Monte Carlo type. In our context, an important question is
how to extend Riesz--Raikov sums and Birkhoff
averages to functions that fail to be
periodic, but possess some repetitivity structure instead.

\section{Averaging almost periodic functions}

At this point, we need to recall some basic definitions and results
from the theory of almost periodic functions in the sense of Bohr
\cite{B-Bohr}, where we refer to \cite[Sec.~8.2]{B-TAO} for a short
summary, to \cite[Sec.~VI.5]{B-Katz} or \cite{B-Cord} for
comprehensive expositions, and to \cite[Sec.~41]{B-Loo} for a more
general and abstract setting (including non-Abelian groups).

Recall that $f\in C (\RR)$ is called \emph{almost periodic in the
  sense of Bohr} if, for any $\varepsilon
> 0$, the set of $\varepsilon$-almost periods
\[
     \cP_{\varepsilon} \, := \, \big\{ t \in \RR :
     \| f - T_{t} f \|^{}_{\infty} < \varepsilon \big\}
\]
is relatively dense in $\RR$. Here,
$\bigl(T_{t} f\bigr) (x) := f(x-t)$ defines the $t$-translate of
$f$. Any continuous periodic function is almost periodic in this
sense, as is any trigonometric polynomial. Any Bohr-almost periodic
function is bounded and uniformly continuous.  In fact, the
$\|.\|^{}_{\infty}$-closure of the (complex) algebra of trigonometric
polynomials is precisely the space of \emph{all} Bohr-almost periodic
functions \cite{B-Bohr}.

For comparison, $f\in C (\RR)$ is called \emph{almost periodic in the
  sense of Bochner} (for
$\|.\|^{}_{\infty}$, to be precise) if the translation orbit $\{ T_{t}
f : t\in\RR\}$ is precompact in the $\|.\|^{}_{\infty}$-topology. The
fundamental relation among these notions can be summarised as follows;
see \cite[Prop.~8.2]{B-TAO} as well as \cite{B-Katz,B-Cord}.

\begin{fact}\label{B-fact:ap-equiv}
   For\/ $f\in C(\RR)$, the following properties are equivalent.
   \begin{enumerate}\itemsep=2pt
   \item $f$ is Bohr-almost periodic, i.e., $\cP_{\varepsilon}$
      is relatively dense for any\/ $\varepsilon > 0$;
    \item $f$ is Bochner-almost periodic for\/ $\| . \|^{}_{\infty}$,
      i.e., the orbit\/ $\{ T_{t} f : t\in\RR \}$ is precompact in the\/
      $\|.\|^{}_{\infty}$-topology;
   \item $f$ is the limit of a sequence of trigonometric polynomials,
      with uniform convergence of the sequence on\/ $\RR$. \qed
   \end{enumerate}
\end{fact}

In view of these relations, we follow \cite{B-Besi} and speak of
\emph{uniformly almost periodic
  functions} from now on when we
refer to this class. If misunderstandings are unlikely, we will drop
the attribute `uniformly'.  Let us elaborate a little on part $(3)$ of
Fact~\ref{B-fact:ap-equiv}.  If $f$ is almost periodic, its
mean
\begin{equation}\label{B-eq:mean-def}
      M(f) \, = \lim_{T\to\infty} \myfrac{1}{2\ts T}
      \int_{a-T}^{a+T} f(x) \dd x
\end{equation}
exists for any $a\in\RR$, is independent of $a$, and the convergence
is uniform in $a$; compare \cite{B-NS} for a more detailed discussion
of this concept.  When we need to emphasise the role of $a$ for more
general types of functions (say without uniformity of the limit in
$a$), we will write $M(f;a)$.

The \emph{Fourier--Bohr coefficient}
of an almost periodic function $f$ at $k\in\RR$ is given by
\[
        a(k) \, = \, M \bigl( \ee^{-2 \pi \ii k (.)} f \bigr) .
\]
It exists for any $k\in\RR$, and differs from $0$ for at most
countably many values of $k$. Any $k\in\RR$ with $a(k) \ne 0$ is
called a \emph{frequency} of $f$. If $\{ k^{}_{\ell} \}$ is the set of
frequencies of $f$, there is a sequence of trigonometric polynomials
of the form
\begin{equation}\label{B-eq:trig-approx}
   P^{(m)} (x) \, = \sum_{\ell=1}^{n^{}_{m}}
   r^{(m)}_{\ell} \, a (k^{}_{\ell}) \, \ee^{2 \pi \ii k^{}_{\ell} x}
\end{equation}
that converge uniformly to $f$ on $\RR$ as $m\to\infty$. Here, the
numbers $r^{(m)}_{\ell}$, which are known as convergence enforcing
numbers, depend on $m$ and $k^{}_{\ell}$, but not on $a(k^{}_{\ell})$,
and can be chosen as rational numbers \cite[Thm.~I.1.24]{B-Cord}.

To approach averages of almost periodic functions, it is thus more
than natural to begin with the averages of trigonometric polynomials.
We formulate the next result for more general sequences than the
exponential ones from above.

\begin{proposition}\label{B-prop:mean-conv}
  Let\/ $P_{m}$ be a $(\nts$complex$\ts\ts)$ trigonometric polynomial of the
  form
  \[   P_{m} (x) \, = \, a^{}_{0} \ts + \sum_{\ell=1}^{m} a^{}_{\ell} \, 
  \ee^{2 \pi \ii k_{\ell} x} \ts ,    \]
  with coefficients\/ $a^{}_{\ell} \in \CC$ and distinct non-zero
  frequencies\/ $k^{}_{1}, \ldots , k^{}_{m}$. Further, let\/
  $(u^{}_{n})^{}_{n\in\NN_{0}}$ be a sequence of real numbers
  such that\/ $\inf_{n\ne m} \, \lvert u^{}_{n} - u^{}_{m} \rvert > 0$.
  Then, for a.e.\/ $x\in\RR$, one has
\[
     \lim_{N\to\infty} \myfrac{1}{N} \sum_{n=0}^{N-1} 
     P_{m} (u^{}_{n} x) \, = \, M (P_{m}) 
     \, = \, a^{}_{0} \ts .
\]
  In particular, this holds for\/ $u^{}_{n} = \alpha^{n}$ with\/
  $\alpha\in\RR$ and\/ $\lvert \alpha \rvert > 1$.
\end{proposition}

\begin{proof}
  The claim is obvious for $m=0$, where the polynomial is
  constant.  The case $m=1$ with $a^{}_{0}=0$, where $P_{m}$ is a
  monomial, is Example~\ref{B-ex:trig-mono} from above. So, for a
  general $P_{m}$, the claim is true for each summand individually,
  with an exceptional set $E(k^{}_{\ell})$ of measure $0$ for $\ell\geq
  1$. Since $\bigcup_{\ell=1}^{m} E(k^{}_{\ell})$ is still a null set,
  the statement on the limit is clear, while its value follows from a
  simple calculation with the mean; compare
  Example~\ref{B-ex:trig-mono}.
\end{proof}

Before we proceed, let us recall the following useful property
of the mean.

\begin{lemma}\label{B-lemma:mean-approx}
  Let\/ $(g^{}_{n})^{}_{n\in\NN}$ be a sequence of complex-valued, but
  not necessarily continuous, functions on\/ $\RR$ that converge
  uniformly to a function\/ $f$. Assume further that the mean\/
  $M(g^{}_{n})$ exists for all\/ $n\in\NN$. Then, also\/ $M(f)$
  exists, and\/ $\, \lim_{n\to\infty} \ts M(g^{}_{n}) = M(f)$. In
  particular, one has
\[
    M(f) \, = \lim_{T\to\infty} \myfrac{1}{2 \ts T}
    \int_{a-T}^{a+T} \! f(x) \dd x
\]
for any fixed\/ $a\in\RR$. When the convergence of the means\/
$M(g^{}_{n}) = M(g^{}_{n}; a)$ is uniform in\/ $a$, then so is
the convergence of\/ $M(f)$.
\end{lemma}

\begin{proof}
  The assumed uniform convergence also means that
  $(g^{}_{n})^{}_{n\in\NN}$ is a Cauchy sequence in the
  $\|.\|^{}_{\infty}$-topology. Fix $\varepsilon>0$ and choose
  $n^{}_{0} = n^{}_{0} (\varepsilon)$ such that
  $\| g^{}_{n} - f \|^{}_{\infty} < \varepsilon$ as well as
  $\| g^{}_{n} - g^{}_{m} \|^{}_{\infty} < \varepsilon$ holds for all
  $n,m \geqslant n^{}_{0}$. Then, for any $T>0$, one has
\begin{equation}\label{B-eq:esti}
   \myfrac{1}{2 \ts T} \biggl| \int_{-T}^{T} 
   \bigl( g^{}_{n} (x) - g^{}_{m} (x)
   \bigr) \dd x \,\biggr|  \, \leqslant \, 
   \big\| g^{}_{n} - g^{}_{m} \big\|_{\infty} \, < \, \varepsilon
\end{equation}
for all $n,m \geqslant n^{}_{0}$, which implies
\[
\begin{split}
   \bigl| M (g^{}_{n}) & - M (g^{}_{m}) \bigr| \,  \leqslant \,
   \big\| g^{}_{n} - g^{}_{m} \big\|_{\infty}  \\[2mm] & 
   + \biggl| M (g^{}_{n}) - \myfrac{1}{2 \ts T} \int_{-T}^{T}
   g^{}_{n} (x) \dd x \, \biggr| + \biggl| 
   M (g^{}_{m}) - \myfrac{1}{2 \ts T} \int_{-T}^{T}
   g^{}_{m} (x) \dd x \, \biggr|  .
\end{split}   
\]
Consequently, $ \bigl| M (g^{}_{n}) - M (g^{}_{m}) \bigr| < 3 \ts
\varepsilon$ for all sufficiently large $T$ due to our assumption on
the existence of the means $M (g^{}_{n})$. Note that, although $T$ may
depend on $m$ and $n$, the above $3\ts\varepsilon$-estimate still works as a
consequence of Eq.~\eqref{B-eq:esti}. The sequence $\bigl( M( g^{}_{n}
) \bigr)_{n\in\NN}$ is thus Cauchy, hence convergent, with limit
$\mathfrak{M}$, say.

Now, choose $n\geqslant n^{}_{0}$ large enough such that also
$\lvert M (g^{}_{n}) - \mathfrak{M} \ts \rvert < \varepsilon$ holds,
fix an arbitrary $a\in\RR$, and consider
\[
\begin{split}
   \biggl| \myfrac{1}{2 \ts T} \int_{a-T}^{a+T}
    \! f (x)  \dd x - \mathfrak{M}
   \, \biggr| \, & \leqslant \, \big\| f - g^{}_{n} \big\|_{\infty}
   + \bigl| M (g^{}_{n}) - \mathfrak{M} \ts \bigr|  \\[1mm]
   & \quad + \biggl| \myfrac{1}{2 \ts T}
   \int_{a-T}^{a+T} \! g^{}_{n} (x) \dd x - M (g^{}_{n}) \, \biggr| 
   \, < \, 3 \ts \varepsilon \ts ,
\end{split}
\]
where the last step holds for all sufficiently large $T$ by
assumption. This derivation implies
\[
     \lim_{T\to\infty} \myfrac{1}{2 \ts T} 
     \int_{a-T}^{a+T} \! f(x) \dd x \, = \, 
     \mathfrak{M} \, =
     \lim_{n\to\infty} M (g^{}_{n}) \ts ,
\]
which is independent of $a\in\RR$, and $M(f) = \mathfrak{M}$ is the
claimed mean of $f$.

When, in addition, the means of the functions $g^{}_{n}$ exist
uniformly in $a$, our $3\ts \varepsilon$-argument also implies
that the convergence of $M(f;a)$ is uniform in $a\in\RR$ as claimed.
\end{proof}

This enables us to formulate the following result.

\begin{theorem}\label{B-thm:Bohr}
  Let\/ $\alpha\in\RR$ with\/ $\lvert \alpha \rvert > 1$ be given,
  and let\/ $f$ be a Bohr-almost periodic function on\/ $\RR$.
  Then, for a.e.\/ $x\in\RR$, one has
\[
    \lim_{N\to\infty} \myfrac{1}{N} \sum_{n=0}^{N-1} 
     f (\alpha^{n} x) \, = \, M (f) \ts . 
\]
\end{theorem}

\begin{proof}
  Let $(g^{}_{n})^{}_{n\in\NN}$ be a sequence of trigonometric
  polynomials that converge uniformly to $f$. As is well known,
  compare \cite{B-Cord}, and is a rather direct consequence of
  Eq.~\eqref{B-eq:trig-approx}, the sequence can be chosen such that
  the frequency sets $\{ k^{}_{j} : 1 \leqslant j \leqslant m^{}_{n} \}$ of the
  $g^{}_{n}$ are nested.  By Proposition~\ref{B-prop:mean-conv}, we
  know that, for every $n\in\NN$,
\[
    \lim_{N\to\infty} \myfrac{1}{N} \sum_{\ell=0}^{N-1}
    g^{}_{n} (\alpha^{\ell} x) \, = \, M (g^{}_{n})
\]
holds for a.e.\ $x\in\RR$, where we denote the excluded null set by
$E_{n}$. By construction, we have $E_{n} \subseteq E_{n+1}$, and
$E:= \bigcup_{n\in\NN} E_{n}$ is still a null set.

Define the Birkhoff average of $\varphi$ at $x$ as $S^{}_{\nts N}
(\varphi, x) = \frac{1}{N} \sum_{n=0}^{N-1} \varphi (\alpha^{n} x)$,
and fix some $\varepsilon > 0$. Choose $n^{}_{0} = n^{}_{0}
(\varepsilon)$ such that $\| f - g^{}_{n} \|^{}_{\infty} <
\varepsilon$ for all $n\geqslant n^{}_{0}$, which is possible under our
assumptions.  Now, for any fixed $x \in \RR \setminus E$, we can
estimate
\[
   \bigl| S^{}_{\nts N} (f, x) - M(f) \bigr| \, \leqslant \,
   \bigl| S^{}_{\nts N} (f \nts - g^{}_{n}, x) \bigr| +
   \bigl| S^{}_{\nts N} (g^{}_{n}, x) - M(g^{}_{n}) \bigr| +
   \bigl| M(g^{}_{n}) - M(f) \bigr|
\]
where, independently of $N$,
\[
   \bigl| S^{}_{\nts N} (f \nts - g^{}_{n}, x) \bigr|
   \, \leqslant \, S^{}_{\nts N} \bigl( \lvert
   f \nts - g^{}_{n} \rvert , x \bigr) \, \leqslant \,
   \| f \nts - g^{}_{n} \|^{}_{\infty} \, < \, \varepsilon
\]
for any $n\geqslant n^{}_{0}$. The third term on the right-hand side of the
previous estimate is smaller than $\varepsilon$ for sufficiently large
$n$ as a consequence of Lemma~\ref{B-lemma:mean-approx}, while the
middle term, under our assumptions, is bounded by $\varepsilon$ for
sufficiently large $N$, which we are still free to choose. This $3\ts
\varepsilon$-argument thus establishes the claim.
\end{proof}

Let us mention in passing that Theorem~\ref{B-thm:Bohr} still 
holds if $\alpha^{n}$, as before, is replaced by the numbers
$u^{}_{n}$ of a sequence as described in 
Remark~\ref{B-rem:gen-u}.\smallskip

At this point, to go any further, we need to extend the class of
functions we consider. This is motivated by the fact that uniform
almost periodicity is often too restrictive. In particular, in various
examples from dynamical systems theory, one encounters 
averages over functions that fail to be bounded, and hence
cannot be uniformly almost periodic. Being unbounded, such
functions cannot be locally Riemann-integrable either, though
they might still admit improper Riemann integrals or
be locally Lebesgue{\ts}-integrable.

It would be natural to investigate the question in the setting of
weakly almost periodic functions, as introduced in \cite{B-NS},
which seems possible as well. However, the above remarks indicate that
one needs results also for functions that violate continuity.  This
suggests to use the wider class of almost periodic functions in the
sense of Stepanov\footnote{The widely used modern version of the name
  is V.V.~Stepanov, while the author used W.~Stepanoff in his original
  articles.}  \cite{B-Step}, which relate to locally
Lebesgue{\ts}-integrable functions like uniform almost periodic
functions do to continous functions. The new norm on
$L^{1}_{\mathrm{loc}} (\RR)$ is given by
\[
    \| f \|^{}_{\mathrm{S}} \, =  \, \sup_{x\in\RR} \myfrac{1}{L}
    \int_{x}^{x+L} \lvert f(y) \rvert \dd y \ts ,
\]
where $L>0$ is an arbitrary, but fixed number. Since these norms are
equivalent for different values of $L$, it is most convenient to
choose $L=1$, as we do from now on. Now, a locally
Lebesgue{\ts}-integrable function $f$ is called \emph{almost periodic
  in the sense of Stepanov}, or
S-almost periodic for short, if, for any $\varepsilon > 0$, the set
$\cP^{\mathrm{S}}_{\varepsilon}$ of $\varepsilon$-almost periods of
$f$ for $\|.\|^{}_{\mathrm{S}}$ is relatively dense. The analogue of
Fact~\ref{B-fact:ap-equiv} then reads as follows (we omit a proof
because it works the same way as in the previous case; 
compare \cite{B-Cord}).

\begin{fact}\label{B-fact:S-equiv}
For\/ $f\in L^{1}_{\mathrm{loc}}(\RR)$, the following properties 
are equivalent.
   \begin{enumerate}\itemsep=2pt
   \item $f$ is\/ $\mathrm{S}$-almost periodic, i.e.,
     $\cP^{\mathrm{S}}_{\varepsilon}$ is relatively dense for any\/
     $\varepsilon > 0$;
   \item $f$ is Bochner-almost periodic for\/ $\|.\|^{}_{\mathrm{S}}$,
     i.e., the orbit\/ $\{ T_{t} f : t\in\RR \}$ is precompact in the\/
     $\|.\|^{}_{\mathrm{S}}$-topology;
   \item $f$ is the\/ $\|.\|^{}_{\mathrm{S}}$-limit of a sequence of
     trigonometric polynomials.  \qed
   \end{enumerate}
\end{fact}

Let us note in passing that every locally integrable function $f$ on
$\RR$ may be viewed as a translation bounded measure (where $f$ is the
Radon--Nikodym density relative to $\lambda$). In doing so, the
Stepanov norm is induced by the $\| . \|^{}_{[0,1]}$-norm for measures as
discussed in \cite{B-NS}. This implies that a function
$f\in L^{1}_{\mathrm{loc}}(\RR)$ is S-almost periodic if and only if
the measure $f\lambda$ is norm-almost periodic in the sense of
\cite{B-Crelle,B-NS}.

Every uniformly almost periodic function is S-almost periodic, which
also means (via part (3) of Fact~\ref{B-fact:S-equiv}) that any
S-almost periodic function can be $\|.\|^{}_{\mathrm{S}}$-approximated
by uniformly almost periodic functions. In other words, the class of
all S-almost periodic functions can equivalently be described as the
$\|.\|^{}_{\mathrm{S}}$-closure of the (complex) algebra of
trigonometric polynomials or as that of the class of uniformly almost
periodic functions. Moreover, the space of \mbox{S-almost} periodic 
functions is complete in the $\|.\|^{}_{\mathrm{S}}$-norm, and $\| f
\|^{}_{\mathrm{S}} =0$ means $f=0$ in the Lebesgue sense, so
$f (x) = 0$ for a.e.\ $x\in\RR$; see \cite{B-BB27} for details.

\begin{remark}
If $f$ is S-almost periodic, its mean exists. In fact,
observe that, for all S-almost periodic
functions $f, g$ and for any
$a\in\RR$, one has
\[
   \myfrac{1}{2 \ts T} \biggl| \int_{a-T}^{a+T} \!
   \bigl( f(x) - g(x) \bigr) \dd x \, \biggr| \; \leqslant \;
   \frac{1 + [2 \ts T]}{2 \ts T} \, \big\| f - g 
   \big\|_{\mathrm{S}} \ts .
\]
Now, it is immediate that the statement of
Lemma~\ref{B-lemma:mean-approx} still holds if uniform convergence is
replaced by $\|.\|^{}_{\mathrm{S}}$-convergence.  This then gives the
desired existence of means because, by
Fact~\ref{B-fact:S-equiv}{\ts}(3), we can
$\|.\|^{}_{\mathrm{S}}$-approximate any \mbox{S-almost} periodic
function with trigonometric polynomials for which the mean clearly
exists.  \exend
\end{remark}

As an aside, we mention the following interesting connection.

\begin{lemma}
  Let\/ $f$ be an\/ $\mathrm{S}$-almost periodic function, and let\/
  $\delta>0$ be arbitrary, but fixed. Then, the function\/
  $f^{}_{\delta}$ defined by
\[
      f^{}_{\delta} (x) \, = \, \myfrac{1}{2\ts \delta} 
     \int_{x-\delta}^{x+\delta} f(y) \dd y
\]   
is continuous and uniformly almost periodic. Moreover,
$\,\lim_{\delta\searrow 0} \, f^{}_{\delta} = f$ in the\/
$\| . \|^{}_{\mathrm{S}}$-topology.
\end{lemma}

\begin{proof}
  Assume $\delta \leqslant \frac{1}{2}$ (the argument for
  $\delta > \frac{1}{2}$ is analogous), and let $t$ be a
  $(2 \ts \delta \ts \varepsilon)$-almost period of $f$ for
  $\|.\|^{}_{\mathrm{S}}$. Now,
\[
\begin{split}
  \bigl| f^{}_{\delta} (t+x) - f^{}_{\delta} (x) \bigr|
  \, & = \, \myfrac{1}{2 \delta} \, \biggl| 
  \int_{x-\delta}^{x+\delta} f (t+y) - f(y) \dd y \, \biggr| \\[2mm]
  \, & \leqslant \, \myfrac{1}{2 \delta}
  \int_{x-\delta}^{x-\delta+1} \bigl| f (t + y) - f (y)
  \bigr| \dd y  \, \leqslant \, 
  \myfrac{ \| f - T_{t} f \|^{}_{\mathrm{S}}}{2 \delta}
  \, < \, \varepsilon
\end{split}
\]
which implies that $t$ is an $\varepsilon$-almost period of
$f^{}_{\delta}$ for $\|.\|^{}_{\infty}$. Via part (1) of
Fact~\ref{B-fact:S-equiv}, we conclude that $f^{}_{\delta}$ satisfies
part (1) of Fact~\ref{B-fact:ap-equiv}, and thus is uniformly almost
periodic. As such, $f^{}_{\delta}$ is also uniformly continuous.

For the second claim, we refer to the original proof in \cite{B-BB27},
which uses an approximation argument that is based on the effect that
a `convolution mollifier' has on a locally \mbox{Lebesgue{\ts}}-integrable
function.
\end{proof}

The main extension of Theorem~\ref{B-thm:p-Sobol} 
can  be stated as follows.

\begin{theorem}\label{B-thm:ap-Sobol}
  Let\/ $\alpha\in\RR$ with\/ $\lvert \alpha \rvert >1$ be fixed, and
  let\/ $f \in L^{1}_{\mathrm{loc}} (\RR)$ be an\/ $\mathrm{S}$-almost
  periodic function.  Assume now
  that there is a uniformly discrete set\/ $\ts Y \! \subset \RR$
  such that\/ $f \!$, for every\/ $\delta > 0$, is locally
  Riemann-integrable on the complement of\/ $Y \nts\nts + (-\delta,\delta)$.
  Assume further that there is a\/ $\delta'>0$ such that, for any\/
  $z\in Y\!$, $f$ is differentiable on the punctured interval\/
  $(z-\delta', z+\delta') \setminus \{ z \}$ and that, for any\/
  $s>0$ and with\/ $V^{}_{\! N} (z,s)$ as defined in
  Theorem~$\ts\ref{B-thm:p-Sobol}$,
\[
    \sup_{z\in Y} V^{}_{\! N} (z, s) \, = \, \scO
    \bigl( N_{\vphantom{I}}^{\frac{s}{2} - \eta} \bigr)
\]
 holds for some\/ $\eta >0$ as\/ $N\to\infty$.

  Then, for a.e.\/ $x\in\RR$, one has
\[
     \lim_{N\to\infty} \myfrac{1}{N} \sum_{n=0}^{N-1} 
     f (\alpha^{n} x) \, = \, M (f) \ts ,
\]
where the mean exists because\/ $f$ is\/
$\mathrm{S}$-almost periodic.
\end{theorem}

\begin{proof}[Sketch of proof]
  Without loss of generality, we may assume that $\delta'$ is small
  enough so that the open intervals $(z-\delta', z+\delta')$ with
  $z\in Y$ are disjoint. Now, Lemma~\ref{B-lem:Dio} guarantees that
  the sequence $(\alpha^{n-1} x)^{}_{n\in\NN}$, for a.e.\ $x\in\RR$,
  does not come closer to $Y$ than $1/n^{1+\varepsilon}$, for any
  fixed $\varepsilon > 0$ and then for all $n\in\NN$ except at most
  finitely many. 

  For any $z\in Y\!$, we have
  $ V^{}_{\! N} (z, 1+\varepsilon) = \scO \bigl(
  N^{\frac{1+\varepsilon}{2} - \eta} \bigr)$
  for some fixed $\eta > 0$ by assumption, where we may once again
  assume that $\varepsilon > 0$ is chosen such that
  $\vartheta = \eta - \frac{\varepsilon}{2} > 0$. With the estimate of
  Eq.~\eqref{B-eq:bound} in the proof of Theorem~\ref{B-thm:p-Sobol},
  we again obtain
  $\cD^{}_{\nts N} \cdot V^{}_{\! N} (z, 1+\varepsilon) = \scO (1)$ as
  $N\to\infty$, which establishes a sufficient criterion for the
  claimed convergence.

  Indeed, let $0<\delta<\delta'$ be arbitrary, and let $1_{\delta}$
  denote the characteristic function of the set
  $\RR \nts \setminus \nts \bigl( Y +
  (-\delta,\delta)\bigr)$.
  Obviously, for any such $\delta$, the function
  $f^{(\delta)} := f \cdot 1_{\delta}$ is both S-almost periodic and
  locally Riemann-integrable on $\RR$. For a.e.\ $x\in\RR$, we thus
  get
\[
   \myfrac{1}{N} \sum_{n=0}^{N-1}  f^{(\delta)}
   (\alpha^n x) \, \xrightarrow{\, N \to \infty \,} \,
   M \bigl( f^{(\delta)} \bigr)
   \, \xrightarrow{\, \delta \ts\ts \searrow \ts 0 \,} \,  M (f)
\]
by a combination of our previous arguments. Since the average of
$f - f^{(\delta)}$ along the exponential sequence is controlled by the
above mentioned estimate from Eq.~\eqref{B-eq:bound}, our claim
follows.
\end{proof}

Note that our assumption on $\eta$ achieves the analogue of
the comment made in Remark~\ref{B-rem:min-eta}. Note also
that Remark~\ref{B-rem:bad} has an obvious extension to this
more general situation. Indeed, one can once again replace
the differentiability condition by the corresponding behaviour
of the total variation in the vicinity of the `bad' points.

\section{Further directions and extensions}

Our exposition so far used complex-valued almost periodic functions
over $\RR$, mainly for ease of presentation. More generally, one is
interested in vector-valued functions, or in function with values in
an arbitrary Banach space $\XX$, with norm
$\boldsymbol{\lvert} . \boldsymbol{\rvert}$ say.  So, let
$f \! : \, \RR \xrightarrow{\quad} \XX$ be such a function, and define
$\| f \|^{}_{\infty} = \sup_{x\in\RR} \, \boldsymbol{\lvert} f(x)
\boldsymbol{\rvert}$.
Then, the $\varepsilon$-almost periods of $f$ are again defined as
\[
     \cP^{}_{\varepsilon} \, := \, \{ t \in \RR :
     \| f - T^{}_{t} f \|^{}_{\infty} < \varepsilon \} \ts ,
\]
with $\bigl( T^{}_{t} f \bigr) (x) = f (x-t)$ as before.

Likewise, one can define trigonometric
polynomials (or functions), by which
one now means any function $T \! : \, \RR \xrightarrow{\quad}
\XX$ of the form
\begin{equation}\label{B-eq:Banach-poly}
     Q_{m} (x) \, = \, a^{}_{0} \, +  \sum_{\ell = 1}^{m} 
     \ee^{2 \pi \ii k^{}_{\ell} x} \, a^{}_{\ell}
\end{equation}
for some $m\geqslant 0$, where $\{ k^{}_{1}, \ldots , k^{}_{\ell} \}$
are distinct, non-zero real numbers and where the $a^{}_{\ell}$ are 
now elements of $\XX$. When $m=0$, the sum is meant to be
empty and $Q_{m}$ is constant. The analogue of 
Fact~\ref{B-fact:ap-equiv} can now be 
stated as follows; see \cite[Ch.~VI]{B-Cord} for details.

\begin{fact}\label{B-fact:Banach}
   Let\/ $(\XX, \boldsymbol{\lvert} . \boldsymbol{\rvert})$ be a
   Banach space. Then,
   for a continuous function\/ $f\! : \, \RR \xrightarrow{\quad} \XX$, 
   the following properties are equivalent.
   \begin{enumerate}\itemsep=2pt
   \item $f$ is Bohr-almost periodic, i.e., $\cP_{\varepsilon}$ is
     relatively dense for any\/ $\varepsilon > 0$;
    \item $f$ is Bochner-almost periodic for\/ $\| . \|^{}_{\infty}$,
      i.e., the orbit\/ $\{ T_{t} f : t\in\RR \}$ is precompact in the\/
      $\|.\|^{}_{\infty}$-topology;
    \item $f$ is the limit of a sequence of trigonometric polynomials,
      with uniform convergence of the sequence on\/ $\RR$. \qed
   \end{enumerate}
\end{fact}

There is no surprise up to this point, and we have gained rather
little.  To continue, we need the notion of the
\emph{mean} of such a function $f$, and also some
generalisation of the Fourier series expansions. For this, we have to
be able to (locally) integrate the function $f$. A natural approach is
provided by \emph{Bochner's integral}
\cite{B-Bochner}, which can be viewed as an extension of the Lebesgue
integral to functions with values in a general Banach space; see
\cite[App.~E]{B-Cohn} as well as \cite[Sec.~V.5]{B-Yoshida} for modern
expositions.

With this extension, most of our previous results remain true, with
the only change that the coefficients $a^{}_{\ell}$ are now elements
of $\XX$ rather than complex numbers. For instance, one has $M (Q_{m})
= a^{}_{0}$ for the trigonometric polynomial of
Eq.~\eqref{B-eq:Banach-poly}, and the analogue of
Proposition~\ref{B-prop:mean-conv} holds without change.  Now, also
the consecutive steps have their natural analogues, and we obtain the
following result.

\begin{theorem}\label{B-thm:Banach}
  Let\/ $(\XX, \boldsymbol{\lvert} . \boldsymbol{\rvert})$ be a Banach
  space, and let\/ $f \! : \, \RR \xrightarrow{\quad} \XX$ be 
  Bohr-almost periodic. Then, for any fixed\/ $\alpha\in\RR$ with\/
  $\lvert \alpha \rvert > 1$, one has
\[
    \lim_{N\to\infty} \myfrac{1}{N} \sum_{n=0}^{N-1} 
     f (\alpha^{n} x) \, = \, M (f) \ts ,
\]
which holds for a.e.\ $x\in\RR$.  \qed
\end{theorem}

The extension to almost periodic functions in the Stepanov sense works
in complete analogy, and we leave further steps in this direction to
the reader.

\section*{Acknowledgements}

MB would like to thank Jean-Pierre Conze, Michael Coons, Uwe Grimm and
Nicolae Strungaru for discussions and helpful comments. Financial
support by the German Research Council (DFG) through CRC 701 is
gratefully acknowledged.

\end{document}